\documentclass[11pt,reqno]{amsart}

\usepackage{graphicx}

\usepackage{xspace}
\usepackage{amsmath,amsthm,amsfonts,amssymb,latexsym,mathrsfs,color,extarrows}
\usepackage{hyperref}

\newtheorem{theorem}{Theorem}
\newtheorem{corollary}[theorem]{Corollary}
\newtheorem{proposition}[theorem]{Proposition}

\newtheorem{lemma}[theorem]{Lemma}
\newtheorem{definition}[theorem]{Definition}

\newcommand{\cpkk}{{\rm cpk\,}}
\newcommand{\cdasc}{{\rm cdasc\,}}
\newcommand{\cddes}{{\rm cddes\,}}

\newcommand{\crun}{{\rm crun\,}}

\newcommand{\altrun}{{\rm altrun\,}}

\newcommand{\ass}{{\rm as\,}}

\newcommand{\udrun}{{\rm udrun\,}}

\newcommand{\fap}{{\rm fap\,}}

\newcommand{\ap}{{\rm ap\,}}
\newcommand{\lap}{{\rm la\,}}

\newcommand{\des}{{\rm des\,}}

\newcommand{\cyc}{{\rm cyc\,}}
\newcommand{\fix}{{\rm fix\,}}

\newcommand{\cc}{{\mathcal C}}

\newcommand{\mdn}{\mathcal{D}}

\newcommand{\msn}{\mathfrak{S}_n}

\newcommand{\ms}{\mathfrak{S}}

\newcommand{\lrf}[1]{\lfloor #1\rfloor}

\newcommand{\mqn}{\mathcal{Q}_n}

\title{The alternating run polynomials of permutations}
\author[S.-M.~Ma]{Shi-Mei Ma}
\address{School of Mathematics and Statistics,
        Northeastern University at Qinhuangdao,
         Hebei 066000, P.R. China}
\email{shimeimapapers@163.com (S.-M. Ma)}
\author[J.~Ma]{Jun Ma}
\address{Department of mathematics, Shanghai Jiao Tong University, Shanghai, P.R. China}
\email{majun904@sjtu.edu.cn(J.~Ma)}
\author[Y.-N. Yeh]{Yeong-Nan Yeh}
\address{Institute of Mathematics,
        Academia Sinica, Taipei, Taiwan}
\email{mayeh@math.sinica.edu.tw (Y.-N. Yeh)}
\begin{document}

\maketitle
\begin{abstract}
In this paper, we first consider a generalization of the David-Barton identity
which relate the alternating run polynomials to Eulerian polynomials.
By using context-free grammars, we then present a combinatorial interpretation of a family of $q$-alternating run polynomials.
Furthermore, we introduce the definition of semi-$\gamma$-positive polynomial and
we show the semi-$\gamma$-positivity of the alternating run polynomials of dual Stirling permutations.~A
connection between the up-down run polynomials of permutations and the alternating run polynomials of dual Stirling permutations is established.
\bigskip

\end{abstract}
\date{\today}
\section{Introduction}
\hspace*{\parindent}
The enumeration of permutations by number of alternating runs was first studied by Andr\'e~\cite{Andre84}.
Knuth~\cite[Section~5.1.3]{Knuth73} has discussed this topic in connection to sorting
and searching. Over the past few decades, the study of alternating runs of permutations was initiated by David and Barton~\cite[157-162]{David62}.

Let $\msn$ denote the symmetric group of all permutations of $[n]=\{1,2,\ldots,n\}$.
Let $\pi=\pi(1)\pi(2)\cdots\pi(n)\in\msn$.
An {\it alternating run} of $\pi$ is a maximal consecutive subsequence that is increasing or decreasing (see~\cite{Andre84,Ma132}).
An {\it up-down run} of $\pi$ is an alternating run of $\pi$ endowed with a 0 in the front (see~\cite{Taylor03,Ma132}).
Let $\altrun(\pi)$ (resp.~$\udrun(\pi)$) be the number of alternating runs (resp.~up-down runs) of $\pi$.
For example, if $\pi=324156$, then  $\altrun(\pi)=4,~\udrun(\pi)=5$.
We define
\begin{align*}
R_{n,k}&=\#\{\pi\in\msn: \altrun(\pi)=k\},\\
T_{n,k}&=\#\{\pi\in\msn: \udrun(\pi)=k\}.
\end{align*}
It is well known that these numbers satisfy the following recurrence relations
\begin{align*}
R_{n+1,k}&=kR_{n,k}+2R_{n,k-1}+(n-k+1)R_{n,k-2},
\end{align*}
\begin{equation}\label{Tnk-recu}
T_{n+1,k}=kT_{n,k}+T_{n,k-1}+(n-k+2)T_{n,k-2},
\end{equation}
with the initial conditions $R_{1,0}=1$ and $R_{1,k}=0$ for $k\geq 1$, $T_{0,0}=1$ and $T_{0,k}=0$ for $k\geq 1$ (see~\cite{Andre84,Taylor03}).
The {\it alternating run polynomial} and {\it up-down run polynomial} are respectively defined by
$R_n(x)=\sum_{k=0}^{n-1}R_{n,k}x^k$ and
$T_n(x)=\sum_{k=0}^nT_{n,k}x^k$.

A {\it descent} of $\pi\in\msn$ is an index $i\in [n-1]$ such that $\pi(i)>\pi(i+1)$. Denote by $\des(\pi)$ the number of descents of $\pi$.
The classical {\it Eulerian polynomial} is defined by
$A_n(x)=\sum_{\pi\in\msn}x^{\des(\pi)+1}$.
By solving a differential equation, David and Barton~\cite[157-162]{David62} established the identity:
\begin{equation}\label{RnxAnx}
R_n(x)=\left(\frac{1+x}{2}\right)^{n-1}(1+w)^{n+1}A_n\left(\frac{1-w}{1+w}\right)
\end{equation}
for $n\geq 2$, where $w=\sqrt{\frac{1-x}{1+x}}$. Using~\eqref{RnxAnx}, B\'ona proved that the polynomial $R_n(x)$ has only real zeros (see~\cite{Bona12}).
Moreover, one can prove that $R_n(x)$ has the zero $x=-1$ with the multiplicity $\lrf{\frac{n}{2}}-1$ by using~\eqref{RnxAnx},
which can also be obtained based on the recurrence relation of $R_n(x)$ (see~\cite{Ma08}).
Motivated by~\eqref{RnxAnx}, Zhuang~\cite{Zhuang17} proved several identities expressing polynomials counting permutations by various descent statistics
in terms of Eulerian polynomials.

Let us now recall another combinatorial interpretation of $T_n(x)$.
An {\it alternating subsequence} of $\pi$ is a subsequence $\pi(i_1)\cdots\pi(i_k)$ satisfying
$$\pi(i_1)>\pi(i_2)<\pi(i_3)>\cdots \pi(i_k),$$
where $i_1<i_2<\cdots<i_k$ (see~\cite{Stanley08}).
Denote by $\ass(\pi)$ the number of terms of the longest alternating subsequence of $\pi$.
By definition, we see that $\ass(\pi)=\udrun(\pi)$. Thus
$$T_n(x)=\sum_{\pi\in\msn}x^{\ass(\pi)}.$$

There has been much recent work related to the numbers $R_{n,k}$ and $T_{n,k}$.
In~\cite{Bona00}, B\'ona and Ehrenborg proved that $R_{n,k}^2\geq R_{n,k-1}R_{n,k+1}$.
Subsequently, B\'ona~\cite[Section 1.3.2]{Bona12} noted that
\begin{equation}\label{TnxRnx}
T_n(x)=\frac{1}{2}(1+x)R_n(x)
\end{equation}
for $n\geq 2$. Set $\rho=\sqrt{1-x^2}$. Stanley~\cite[Theorem 2.3]{Stanley08} showed that
\begin{equation}\label{EGF-Tnx}
T(x,z)=:\sum_{n=0}^\infty T_n(x)\frac{z^n}{n!}=(1-x)\frac{1+\rho+2xe^{\rho z}+(1-\rho)e^{2\rho z}}{1+\rho-x^2+(1-\rho-x^2)e^{2\rho z}}.
\end{equation}
By using~\eqref{TnxRnx} and~\eqref{EGF-Tnx}, Stanley~\cite{Stanley08} obtained explicit formulas of $T_{n,k}$ and $R_{n,k}$.
Canfield and Wilf~\cite{Canfield08} presented an asymptotic formula for $R_{n,k}$.
In~\cite{Ma12}, another explicit formula of $R_{n,k}$ was obtained by combining the derivative polynomials of tangent function and
the following generating function obtained by Carlitz~\cite{Carlitz78}:
\begin{equation*}\label{Carlitz}
\sum_{n=0}^\infty\frac{z^n}{n!}\sum_{k=0}^nR_{n+1,k}x^{n-k}=\frac{1-x}{1+x}\left(\frac{\sqrt{1-x^2}+\sin(z\sqrt{1-x^2})}{x-\cos(z\sqrt{1-x^2})}\right)^2.
\end{equation*}
In~\cite{Ma132}, several convolution formulas of the polynomials $R_n(x)$ and $T_n(x)$ are obtained by using Chen's grammars.
By generalizing a reciprocity formula of Gessel, Zhuang~\cite{Zhuang16}
obtained generating function for permutation statistics that are expressible in terms of alternating runs. Very recently, Josuat-Verg\`es and Pang~\cite{Josuat18} showed that alternating runs can be
used to define subalgebras of Solomon's descent algebra.

In this paper,
we continue the work initiated by David and Barton~\cite{David62}.
In Section~\ref{Section:02},
we consider a generalization of~\eqref{RnxAnx}.
In Section~\ref{Section:03}, we present a combinatorial interpretation of a family of $q$-alternating run polynomials by using Chen's grammars. In Section~\ref{semi-gamma}, we show
the semi-$\gamma$-positivity of the alternating run polynomials of dual Stirling permutations.
\section{The David-Barton type identity}\label{Section:02}
Let $f(x)=\sum_{i=0}^nf_ix^i$ be a symmetric polynomial, i.e., $f_i=f_{n-i}$ for any $0\leq i\leq n$. Then $f(x)$ can be expanded uniquely as
$$f(x)=\sum_{k=0}^{\lrf{\frac{n}{2}}}\gamma_kx^k(1+x)^{n-2k},$$ and it is said to be {\it $\gamma$-positive} if $\gamma_k\geq 0$ for $0\leq k\leq \lrf{\frac{n}{2}}$ (see~\cite{Gal05}).
The $\gamma$-positivity provides an approach to study symmetric and unimodal polynomials and has been extensively studied
(see~\cite{Athanasiadis17,Branden08,Chow08,Lin15} for instance).

The first main result of our paper is the following, which shows that the David-Barton type identities often occur in combinatorics and geometry.
\begin{theorem}\label{MnxThm}
Let $$M_n(x)=\sum_{k=0}^{\lrf{(n+\delta)/2}}M(n,k)x^k(1+x)^{n+\delta-2k}$$ be a symmetric polynomial, where $\delta$ is a fixed integer.
Set $w=\sqrt{\frac{1-x}{1+x}}$. Then
\begin{equation}\label{NnxMnx01}
N_n(x)=\left(\frac{1+x}{2}\right)^{n-\delta}(1+w)^{n+\delta}M_n\left(\frac{1-w}{1+w}\right)
\end{equation}
\end{theorem}
if and only if
\begin{equation}\label{NnxMnx02}
N_n(x)=\sum_{k=0}^{\lrf{(n+\delta)/2}}\frac{1}{2^{k-2\delta}}M(n,k) x^k(1+x)^{n-\delta-k}.
\end{equation}
\begin{proof}
Set $\alpha=\frac{1+x}{2}$.
Note that
\begin{align*}
1-w^2&=\frac{x}{\alpha},\\
\frac{1-w}{1+w}&=\frac{1-w^2}{(1+w)^2}=\frac{1}{(1+w)^2}\frac{x}{\alpha},\\
1+\frac{1-w}{1+w}&=\frac{2}{1+w}.
\end{align*}
It follows from~\eqref{NnxMnx01} that
\begin{align*}
N_n(x)&=\left(\frac{1+x}{2}\right)^{n-\delta}(1+w)^{n+\delta}M_n\left(\frac{1-w}{1+w}\right)\\
&=\alpha^{n-\delta}(1+w)^{n+\delta}\sum_k M(n,k)\frac{1}{(1+w)^{2k}}\frac{x^k}{\alpha^k}\left(\frac{2}{1+w}\right)^{n+\delta-2k}\\
&=\sum_k M(n,k) x^k\alpha^{n-\delta-k}2^{n+\delta-2k}\\
&=\sum_k M(n,k) x^k\left(\frac{1+x}{2}\right)^{n-\delta-k}2^{n+\delta-2k}\\
&=\sum_k \frac{1}{2^{k-2\delta}}M(n,k) x^k(1+x)^{n-\delta-k},
\end{align*}
and vice versa. This completes the proof.
\end{proof}

The reader is referred to~\cite{Athanasiadis17} for a survey of some recent results on $\gamma$-positivity.
For any $\gamma$-positive polynomial $M_n(x)$,
we can define an associated polynomial $N_n(x)$ by using~\eqref{NnxMnx02}. And then we get a David-Barton type identity~\eqref{NnxMnx01}.
As illustrations,
in the rest of this section, we shall present two examples.

For example,
Foata and Sch\"utzenberger~\cite{Foata70} discovered that
\begin{equation*}\label{Anx-gamma}
A_n(x)=\sum_{k=1}^{\lrf{({n+1})/{2}}}a(n,k)x^k(1+x)^{n+1-2k}
\end{equation*}
for $n\geq 1$, where the numbers $a(n,k)$ satisfy the recurrence relation
\begin{equation*}\label{ank-recu}
a(n,k)=ka(n-1,k)+(2n-4k+4)a(n-1,k-1),
\end{equation*}
with the initial conditions $a(1,1)=1$ and $a(1,k)=0$ for $k\neq1$ (see~\cite{Chow08,Petersen15} for instance).
By using the David-Barton identity~\eqref{RnxAnx} and Theorem~\ref{MnxThm}, we immediately get the following result.
\begin{proposition}
For $n\geq 2$, we have
$$R_n(x)=\sum_{k=1}^{\lrf{({n+1})/{2}}}\frac{1}{2^{k-2}}a(n,k)x^k(1+x)^{n-1-k}.$$
\end{proposition}

Let $\pm[n]=\{\pm1,\pm2,\ldots,\pm n\}$.
Let $B_n$ be the hyperoctahedral group of rank $n$.
Elements of $B_n$ are signed permutations of $\pm[n]$ with the property that $\pi(-i)=-\pi(i)$ for all $i\in [n]$.
In the sequel, we always assume that signed permutations in $B_n$ are prepended by 0. That is, we identify a signed permutation
$\pi=\pi(1)\cdots \pi(n)$ with the word $\pi(0)\pi(1)\cdots \pi(n)$, where $\pi(0)=0$.
A type $B$ descent is an index $i\in \{0,1,\ldots,n-1\}$ such that $\pi(i)>\pi(i+1)$. Let $\des^B(\pi)$ be the number of type $B$ descents of $\pi$.
The {\it type $B$ Eulerian polynomials} are defined by
$$B_n(x)=\sum_{\pi\in B_n}x^{\des_B(\pi)}.$$
It is well known that
$$B_n(x)=\sum_{k=0}^{\lrf{{n}/{2}}}b(n,k)x^k(1+x)^{n-2k},$$
where the numbers $b(n,k)$ satisfy the recurrence relation
\begin{equation}\label{bnk-recu}
b(n,k)=(1+2k)b(n-1,k)+4(n-2k+1)b(n-1,k-1),
\end{equation}
with the initial conditions $b(1,0)=1$ and $b(1,k)=0$ for $k\neq 0$ (see~\cite{Athanasiadis17,Chow08,Petersen15}).

Define
\begin{equation}\label{bnxdef}
b_n(x)=\sum_{k=0}^{\lrf{{n}/{2}}}\frac{1}{2^k}b(n,k)x^k(1+x)^{n-k}.
\end{equation}
Then by Theorem~\ref{MnxThm}, we get the following result.
\begin{proposition}
For $n\geq 1$, we have
\begin{equation*}
b_n(x)=\left(\frac{1+x}{2}\right)^{n}(1+w)^{n}B_n\left(\frac{1-w}{1+w}\right).
\end{equation*}
\end{proposition}

Combining~\eqref{bnk-recu} and~\eqref{bnxdef}, we see that the polynomials $b_n(x)$ satisfy the recurrence relation
\begin{equation}\label{bnkx-recu}
b_{n+1}(x)=(1+x+2nx^2)b_n(x)+2x(1-x^2)b_n'(x),
\end{equation}
with the initial conditions $b_0(x)=1,~b_1(x)=1+x$. For $n\geq 1$, we
define $b_n(x)=\frac{1+x}{x}c_n(x)$.
It follows from~\eqref{bnkx-recu} that the polynomials $c_n(x)$ satisfy the recurrence relation
\begin{equation*}
c_{n+1}(x)=(2nx^2+3x-1)c_n(x)+2x(1-x^2)c_n'(x).
\end{equation*}
Let $\widehat{B}_n=\{\pi\in B_n \mid \pi(1)>0\}$.
There is a combinatorial interpretation of $c_n(x)$ (see~\cite{Chow14,Zhao11}):
$$c_n(x)=\sum_{\pi\in \widehat{B}_n}x^{\altrun(\pi)}.$$
\section{The $q$-alternating runs polynomials}\label{Section:03}
For an alphabet $A$, let $\mathbb{Q}[[A]]$ be the rational commutative ring of formal power
series in monomials formed from letters in $A$. A {\it Chen's grammar} (which is known as context-free grammar) over
$A$ is a function $G: A\rightarrow \mathbb{Q}[[A]]$ that replaces a letter in $A$ by an element of $\mathbb{Q}[[A]]$, see~\cite{Chen93,Chen17,Ma19} for details.
The formal derivative $D:=D_G$ is a linear operator defined with respect to a context-free grammar $G$.
Following~\cite{Chen17}, a {\it grammatical labeling} is an assignment of the underlying elements of a combinatorial structure
with variables, which is consistent with the substitution rules of a grammar.

Let us now recall two results
on context-free grammars.
\begin{proposition}[{\cite[Theorem~6]{Ma132}}]\label{Ma131}
If $G=\{a\rightarrow ab,~b\rightarrow bc,~c\rightarrow b^2\}$, then
\begin{align*}
D^n(a)&=a\sum_{k=0}^nT_{n,k}b^kc^{n-k},~
D^n(a^2)=a^2\sum_{k=0}^nR_{n+1,k}b^kc^{n-k}.
\end{align*}
\end{proposition}
\begin{proposition}[{\cite[Theorem~9]{Ma132}}]\label{Ma132}
If $G=\{a\rightarrow 2ab,~b\rightarrow bc,~c\rightarrow b^2\}$, then
\begin{align*}
D^n(a)&=a\sum_{k=0}^nR_{n+1,k}b^kc^{n-k}.
\end{align*}
\end{proposition}

Combining Leibniz's formula and Proposition~\ref{Ma131}, we see that
$$R_{n+1}(x)=\sum_{k=0}^n\binom{n}{k}T_k(x)T_{n-k}(x).$$
Motivated by Propositions~\ref{Ma131} and~~\ref{Ma132}, it is natural to consider the grammar
\begin{equation}\label{qxy-grammar}
G_1=\{a\rightarrow qab,~b\rightarrow bc,~c\rightarrow b^2\}.
\end{equation}
Note that
$D_{G_1}(a)=qab,~D_{G_1}^2(a)=a(q^2b^2+qbc)$.
By induction, it is easy to verify that
\begin{equation}\label{DnaRnk}
D_{G_1}^n(a)=a\sum_{k=0}^nR_{n,k}(q)b^kc^{n-k}.
\end{equation}
It follows from~\eqref{qxy-grammar} that
\begin{align*}
  D_{G_1}^{n+1}(a)&=D_{G_1}\left(a\sum_{k=0}^nR_{n,k}(q)b^kc^{n-k}\right)\\
  &=a\sum_{k}R_{n,k}(q)\left(kb^kc^{n-k+1}+qb^{k+1}c^{n-k}+(n-k)b^{k+2}c^{n-k-1}\right),
\end{align*}
which leads to the recurrence relation
\begin{equation}\label{Rnkq-recu}
R_{n+1,k}(q)=kR_{n,k}(q)+qR_{n,k-1}(q)+(n-k+2)R_{n,k-2}(q).
\end{equation}
The {\it $q$-alternating run polynomials} are defined by
$$R_n(x;q)=\sum_{k=0}^nR_{n,k}(q)x^k.$$
In particular, $R_n(x;1)=T_n(x)$, $R_n(x;2)=R_{n+1}(x).$
The first few $R_n(x;q)$ are given as follows:
\begin{align*}
R_0(x;q)=1,~
 R_1(x;q)=qx,~
 R_2(x;q)= qx(1 + q x),~
  R_3(x;q)= qx(1 + 3 q x + x^2 + q^2 x^2).
\end{align*}

We define
$$R(x,z;q):=\sum_{n=0}^{\infty}R_n(x;q)\frac{z^n}{n!}.$$
\begin{proposition}\label{prop-Rxzq}
We have
$R(x,z;q)=T^q(x,z)$,
where $T(x,z)$ is given by~\eqref{EGF-Tnx}.
Therefore,
\begin{equation}\label{Dna-EGF}
\sum_{n=0}^\infty D_{G_1}^n(a)\frac{z^n}{n!}=aR\left(\frac{b}{c},cz;q\right)=aT^q\left(\frac{b}{c},cz\right).
\end{equation}
Moreover, we have $R_n(x;-q)=R_n(-x;q)$ and $R_n(-x;-q)=R_n(x;q)$.
\end{proposition}
\begin{proof}
By rewriting~\eqref{Rnkq-recu} in terms of generating function $R(x,z;q)$, we obtain
\begin{equation}\label{PDE}
(1-x^2z)\frac{\partial}{\partial z}R(x,z;q)=x(1-x^2)\frac{\partial}{\partial x}R(x,z;q)+qxR(x,z;q).
\end{equation}
It is routine to check that the generating function $T^q(x,z)$
satisfies~\eqref{PDE}. Also, this generating function gives $T^q(0,z)=T^q(x,0)=1$.
Hence $R(x,z;q)=T^q(x,z)$.
It is routine to check that
\begin{equation*}\label{Rxzq-symmetric}
R(x,z;-q)=R(-x,z;q),~R(-x,z;-q)=R(x,z;q)
\end{equation*}
which leads to the desired result.
\end{proof}

We say that $\pi\in\msn$ is a
circular permutation if it has only one cycle. Let $A=\{x_1,x_2,\ldots, x_k\}$
be a finite set of positive integers,
and let $\cc_A$ be the set of all circular
permutations of $A$. We will write a permutation $w\in\cc_A$ by using its
canonical presentation $w=y_1y_2\cdots y_k$, where $y_1=\min A, y_i=w^{i-1}(y_1)$ for $2\leq i\leq k$ and $y_1=w^k(y_1)$.
A {\it cycle peak} (resp.~{\it cycle double ascent}, {\it cycle double descent}) of $w$ is an entry $y_i$, $2\leq i\leq k$, such that
$y_{i-1}<y_i>y_{i+1}$ (resp.~$y_{i-1}<y_i<y_{i+1}$, $y_{i-1}>y_i>y_{i+1}$), where we set $y_{k+1}=\infty$.
Let $\cpkk(w)$ (resp.~$\cdasc(w)$,~$\cddes(w)$, $\cyc(w)$) be the number of cycle peaks (resp.~cycle double ascents, cycle double descents, cycles) of $w$.
\begin{definition}
A cycle run of a circular permutation $w$ is an alternating run of $w$ endowed with a $\infty$ in the end.
Let $\crun(w)$ be the number of {\it cycle runs} of $w$.
\end{definition}

It is clear that $\crun(w)=2\cpkk(w)+1$.
In the following discussion we always write $\pi\in\msn$ in standard cycle decomposition: $\pi=w_1\cdots w_k$, where the cycles are written in increasing order of their smallest entry and each of these cycles is expressed in canonical presentation.
We define
$$\crun(\pi):=\sum_{i=1}^k\crun(w_i).$$
In particular, $\crun((1)(2)\cdots(n))=\sum_{i=1}^n\crun(i)=\sum_{i=1}^n\altrun(i\infty)=n$.
We can now present the second main result.
\begin{theorem}
For $n\geq 1$, we have
\begin{equation}\label{Rnxq-des}
R_n(x;q)=\sum_{\pi\in\msn}x^{\crun(\pi)}q^{\cyc(\pi)}.
\end{equation}
\end{theorem}
\begin{proof}
For $\pi\in\msn$, we first put a $\infty$ in the end of each cycle.
We then introduce a grammatical labeling of $\pi$ as follows:
\begin{itemize}
  \item [\rm ($L_1$)] Put a subscript label $q$ at the end of each cycle of $\pi$;
 \item [\rm ($L_2$)] Put a superscript label $a$ at the end of $\pi$;
\item [\rm ($L_3$)] Put a superscript label $b$ before each $\infty$;
\item [\rm ($L_4$)] If $\pi(i)$ is a cycle peak, then put a superscript label $b$ before $\pi(i)$ and a superscript label $b$ right after $\pi$;
\item [\rm ($L_5$)] If $\pi(i)$ is a cycle double ascents, then put the superscript label $c$ before $\pi(i)$;
\item [\rm ($L_6$)] If $\pi(i)$ is a cycle double descents, then put the superscript label $c$ right after $\pi(i)$.
\end{itemize}
The weight of $\pi$ is the product of its labels.
When $n=1,2$, we have $$\ms_1=\{(1^b\infty)_q^a\},~\ms_2=\{(1^b\infty)_q(2^b\infty)^a_q,~(1^c2^b\infty)_q^a\}.$$
Then the weight of $(1^b)_q^a$ is given by $D_{G_1}(a)$, and the sum of weights of the elements in $\ms_2$ is given by $D_{G_1}^2(a)$.
Hence the result holds for $n=1,2$.
Let $$r_n(i,j)=\{\pi\in\msn: \crun(\pi)=i,~\cyc(\pi)=j\}.$$
Suppose we get all labeled permutations in $r_{n-1}(i,j)$, where $n\geq 3$. Let
$\pi'$ be obtained from $\pi\in r_{n-1}(i,j)$ by inserting the entry $n$.
We distinguish the following four cases:
\begin{itemize}
  \item [\rm ($c_1$)] If we insert $n$ as a new cycle, then $\pi'\in r_{n-1}(i+1,j+1)$. This case corresponds to the substitution rule $a\rightarrow qab$.
 \item [\rm ($c_2$)] If we insert $n$ before a $\infty$, then $\pi'\in r_{n-1}(i,j)$.
 This case corresponds to the substitution rule $b\rightarrow bc$;
\item [\rm ($c_3$)] If we insert $n$ before or right after a cycle peak, then $\pi'\in r_{n-1}(i,j)$. This case corresponds to the substitution rule $b\rightarrow bc$;
\item [\rm ($c_4$)] If we insert $n$ before a cycle double ascents or right after a cycle double descents, then $\pi'\in r_{n-1}(i+2,j)$. This case corresponds to the substitution rule $c\rightarrow b^2$.
\end{itemize}
In each case, the insertion of $n$ corresponds to one substitution rule in the grammar~\eqref{qxy-grammar}.
It is easy to check that the action of $D_{G_1}$ on elements of $\ms_{n-1}$ generates all elements of $\ms_n$. Using~\eqref{DnaRnk} and by induction, we present a constructive proof of~\eqref{Rnxq-des}.
This completes the proof.
\end{proof}

%
%
We define
$$R_n(x,y;q)=\sum_{\pi\in\msn}x^{{\crun}(\pi)}y^{\fix(\pi)}q^{\cyc(\pi)},$$
$$R(x,y,z;q)=\sum_{n=0}^\infty R_n(x,y;q)\frac{z^n}{n!}.$$
By using the principle of inclusion-exclusion, it is routine to verify that
$$R_n(x,y;q)=\sum_{i=0}^n\binom{n}{i}(qxy-qx)^iR_{n-i}(x;q).$$
Hence
\begin{equation}\label{Rxyzq}
R(x,y,z;q)=e^{qx(y-1)z}R(x,z;q)=e^{qx(y-1)z}T^q(x,z).
\end{equation}
%
%

A permutation $\pi\in\msn$ is a {\it derangement} if $\pi(i)\neq i$ for any $i\in [n]$. Let $\mdn_n$ denote the set of derangements in $\msn$. Then
$$R_n(x,0;1)=\sum_{\pi\in\mdn_n}x^{{\crun}(\pi)}.$$

\begin{proposition}
Set $d_n(x)=R_n(x,0;1)$. Then the polynomials $d_n(x)$ satisfy the recurrence
\begin{equation}\label{dnx-recu}
d_{n+1}(x)=nx^2d_n(x)+x(1-x^2)d_n'(x)+nxd_{n-1}(x),
\end{equation}
with the initial conditions $d_0(x)=1,~d_1(x)=0$. In particular, $d_n(-1)=-(n-1)$ for $n\geq 1$.
\end{proposition}
\begin{proof}
Let $d(x,z)=\sum_{n=0}^\infty d_n(x)\frac{z^n}{n!}$.
It follow from~\eqref{Rxyzq} that
\begin{equation}\label{dxz-Txz}
d(x,z)=e^{-xz}T(x,z).
\end{equation}
By rewriting~\eqref{Tnk-recu} in terms of generating function $T(x,z)$, we obtain
$$(1-x^2z)\frac{\partial}{\partial z}T(x,z)=xT(x,z)+x(1-x^2)\frac{\partial}{\partial x}T(x,z).$$
Hence
$$(1-x^2z)\frac{\partial}{\partial z}d(x,z)=xzd(x,z)+x(1-x^2)\frac{\partial}{\partial x}d(x,z),$$
which yields the desired recurrence relation.
\end{proof}

Let $d_n(x)=\sum_{k=0}^nd_{n,k}x^k$. By using~\eqref{dxz-Txz}, it is not hard to verify that
$$\sum_{n=0}^\infty d_{n,n}\frac{z^n}{n!}=\frac{e^{-x}}{\tan x+\sec x}.$$

\section{Semi-$\gamma$-positive polynomials}\label{semi-gamma}

Let $g(x)=\sum_{i=0}^{2n}g_ix^i$ be a symmetric polynomial. Note that
\begin{align*}
g(x)&=\sum_{i=0}^{n}\gamma_ix^i(1+x)^{2(n-i)}\\
&=\sum_{i=0}^{n}\gamma_ix^i(1+2x+x^2)^{n-i}\\
&=\sum_{i=0}^{n}\sum_{\ell=0}^{n-i}\binom{n-i}{\ell}2^{\ell}\gamma_i x^{i+\ell}(1+x^2)^{n-i-\ell}.
\end{align*}
Hence $g(x)$ can be expanded as $$g(x)=\sum_{k=0}^{n}\lambda_kx^k(1+x^2)^{n-k}.$$ It is clear that if $\gamma_i\geq 0$ for all $0\leq i\leq n$, then
$\lambda_k\geq 0$ for all $0\leq k\leq n$. Furthermore,
we have
\begin{align*}
g(x)&=\sum_{k=0}^{\lrf{n/2}}\lambda_{2k}x^{2k}(1+x^2)^{n-2k}+\sum_{k=0}^{\lrf{(n-1)/2}}\lambda_{2k+1}x^{2k+1}(1+x^2)^{n-2k-1}\\
&=g_1(x^2)+xg_2(x^2).
\end{align*}

Similarly, if $h(x)=\sum_{i=0}^{2n+1}h_ix^i$ a symmetric polynomial, then we have
\begin{align*}
h(x)&=\sum_{i=0}^{n}\beta_ix^i(1+x)^{2n+1-2i}\\
&=(1+x)\sum_{i=0}^{n}\sum_{\ell=0}^{n-i}\binom{n-i}{\ell}2^{\ell}\beta_i x^{i+\ell}(1+x^2)^{n-i-\ell}.
\end{align*}
Hence $h(x)$ can be expanded as $$h(x)=(1+x)\sum_{k=0}^{n}\mu_kx^k(1+x^2)^{n-k}.$$

\begin{definition}
If $f(x)=(1+x)^\nu\sum_{k=0}^{n}\lambda_kx^k(1+x^2)^{n-k}$ and $\lambda_k\geq 0$ for all $0\leq k\leq n$, then we say that $f(x)$ is semi-$\gamma$-positive, where $\nu=0$ or $\nu=1$.
\end{definition}
It should be noted that a semi-$\gamma$-positive polynomial is not always $\gamma$-positive.
From the above discussion it follows that we have the following result.
\begin{proposition}
If $f(x)=(1+x)^\nu \left(f_1(x^2)+xf_2(x^2)\right)$ is a semi-$\gamma$-positive polynomial, then both $f_1(x)$ and $f_2(x)$ are $\gamma$-positive.
\end{proposition}

In the following, we shall show the semi-$\gamma$-positivity of the alternating run polynomials of dual Stirling permutations.
Following~\cite{Gessel78}, a {\it Stirling permutation} of order $n$ is a permutation of the multiset $\{1,1,\ldots,n,n\}$ such that
for each $i$, $1\leq i\leq n$, all entries between the two occurrences of $i$ are larger than $i$.
There has been much recent work on Stirling permutations, see~\cite{Haglund12,Ma19} and references therein.

Denote by $\mqn$ the set of {\it Stirling permutations} of order $n$.
Let $\sigma=\sigma_1\sigma_2\cdots\sigma_{2n}\in\mqn$.
Let $\Phi$ be the injection which maps each first occurrence of entry $j$ in $\sigma$ to $2j$ and the
second $j$ to $2j-1$,
where $j\in [n]$. For example, $\Phi(221331)=432651$.
Let $\Phi(\mqn)=\{\pi\mid \sigma\in\mqn, \Phi(\sigma)=\pi\}$
be the set of {\it dual Stirling permutations} of order $n$.
Clearly, $\Phi(\mqn)$ is a subset of $\ms_{2n}$.
For $\pi\in\Phi(\mqn)$, the entry $2j$ is to the left of $2j-1$, and all entries in $\pi$ between $2j$ and $2j-1$ are larger than $2j$, where $1\leq j\leq n$.
Noted that $\pi\in \Phi(\mqn)$ always ends with a descending run.
The alternating runs polynomials of dual Stirling permutations are defined by
$$F_n(x)=\sum_{\sigma\in\Phi(\mqn)}x^{\altrun(\sigma)}=\sum_{k=1}^{2n-1}F_{n,k}x^k.$$

According to~\cite{Mawang16}, the numbers $F_{n,k}$ satisfy the recurrence relation
\begin{equation}\label{Fnk-recurrence}
F_{n+1,k}=kF_{n,k}+F_{n,k-1}+(2n-k+2)F_{n,k-2}.
\end{equation}
with the initial conditions $F_{0,0}=1,~F_{1,1}=1$ and $F_{n,0}=0$ for $n\geq1$.
It follows from~\eqref{Fnk-recurrence} that
\begin{equation*}
F_{n+1}(x)=(x+2nx^2)F_n(x)+x(1-x^2)F_n'(x).
\end{equation*}
The first few $F_n(x)$ are given as follows:
\begin{align*}
  F_1(x)&=x, \\
  F_2(x)&=x+x^2+x^3, \\
  F_3(x)&=x+3x^2+7x^3+3x^4+x^5,\\
  F_4(x)&=x+7x^2+29x^3+31x^4+29x^5+7x^6+x^7.
\end{align*}

Let $$r(x)=\sqrt{\frac{1+x}{1-x}}.$$
By induction, it is to verify that
\begin{align*}
&\left(x\frac{d}{dx}\right)^{2n}r(x)=\frac{r(x)F_{2n}(x)}{(1-x^2)^{2n}},\\
&\left(x\frac{d}{dx}\right)^{2n+1}r(x)= \frac{F_{2n+1}(x)}{r(x)(1-x^2)^{2n}(1-x)^2}.
\end{align*}
\begin{lemma}[\cite{Mawang16}]\label{Prop2}
If
\begin{equation}\label{Grammar-fap}
G_2=\{x\rightarrow xyz, y\rightarrow yz^2, z\rightarrow y^2z\},
\end{equation}
then we have
\begin{equation}\label{DnxFnk}
D_{G_2}^n(x)=x\sum_{\sigma\in\Phi(\mqn)}y^{\altrun(\sigma)}z^{2n-\altrun(\sigma)}=x\sum_{k=0}^{2n-1}F_{n,k}y^kz^{2n-k}.
\end{equation}
\end{lemma}

We now recall another combinatorial interpretation of $F_{n}(x)$.
An occurrence of an {\it ascent-plateau} of $\sigma\in\mqn$ is an index $i$ such that $\sigma_{i-1}<\sigma_{i}=\sigma_{i+1}$, where $i\in\{2,3,\ldots,2n-1\}$. An occurrence of a {\it left ascent-plateau} is an index $i$ such that $\sigma_{i-1}<\sigma_{i}=\sigma_{i+1}$, where $i\in\{1,2,\ldots,2n-1\}$ and $\sigma_0=0$.
Let $\ap(\sigma)$ and $\lap(\sigma)$ be the numbers of ascent-plateaus and left ascent-plateaus of $\sigma$, respectively.
The number of flag ascent-plateaus of $\sigma$ is defined by $$\fap(\sigma)=\left\{
               \begin{array}{ll}
                 2\ap(\sigma)+1, & \hbox{if $\sigma_1=\sigma_2$;} \\
                 2\ap(\sigma), & \hbox{otherwise.}
               \end{array}
             \right.
$$
Clearly, $\fap(\sigma)=\ap(\sigma)+\lap(\sigma)$.
Following~\cite[Section~3]{Ma19}, we have
\begin{equation*}\label{Dnx}
D_{G_2}^n(x)=x\sum_{\sigma\in\mqn}y^{\fap(\sigma)}z^{2n-\fap(\sigma)}.
\end{equation*}
Thus,
$$F_n(x)=\sum_{\sigma\in\mqn}x^{\fap(\sigma)}.$$
In fact, it is easy to verify that $\fap(\sigma)=\altrun( \Phi(\sigma))$ for any $\sigma\in\mqn$.

\begin{proposition}\label{propFnx}
For $n\geq 1$, we have
\begin{equation*}
F_n(x)=\sum_{k=1}^n\gamma_{n,k}x^k(1+x)^{2n-2k},
\end{equation*}
where the numbers $\gamma_{n,k}$ satisfy the recurrence relation
\begin{equation}\label{gnk-recur}
\gamma_{n+1,k}=k\gamma_{n,k}+(2n-4k+5)\gamma_{n,k-1},
\end{equation}
with the initial conditions $\gamma_{1,1}=1$ and $\gamma_{1,k}=0$ for $k\neq 1$. In particular, $$\gamma_{n+1,n+1}=(-1)^n(2n-1)!!~{\text{for $n\geq 1$}}.$$
\end{proposition}
\begin{proof}
We first consider a change of the grammar~\eqref{Grammar-fap}. Set $a=yz$ and $b=y+z$. Then we have
$D(x)=xa,~D(a)=a(b^2-2a),~D(b)=ab$.
If $$G_3=\{x\rightarrow xa,~a\rightarrow a(b^2-2a),~b\rightarrow ab\},$$ then
by induction, we see that there exist integers $\gamma_{n,k}$ such that
\begin{equation}\label{G3def}
D_{G_3}^n(x)=x\sum_{k=0}^n\gamma_{n,k}a^kb^{2n-2k}.
\end{equation}
Note that
\begin{align*}
D_{G_3}^{n+1}(x)
&=D_{G_3}\left(x\sum_{k=1}^n\gamma_{n,k}a^kb^{2n-2k}\right)\\
&=x\sum_k \gamma_{n,k}a^kb^{2n-2k}\left(a+kb^{2}-2ka+(2n-2k)a\right)
\end{align*}
By comparing the coefficients of $a^kb^{2n-2k+2}$, we immediately get~\eqref{gnk-recur}. Moreover, it is clear that $\gamma_{n,0}=0$ for $n\geq 1$.
By using~\eqref{G3def}, upon taking $a=yz$ and $b=y+z$, we get
\begin{equation}\label{DnxFnk01}
D_{G_2}^n(x)=x\sum_{k=0}^n\gamma_{n,k}(yz)^k(y+z)^{2n-2k}.
\end{equation}
Then comparing~\eqref{DnxFnk01} with~\eqref{DnxFnk}, we see that $F_n(x)=\sum_{k=1}^n\gamma_{n,k}x^k(1+x)^{2n-2k}$ for $n\geq 1$.
By using~\eqref{gnk-recur}, we obtain
$$\gamma_{n+1,n+1}=-(2n-1)\gamma_{n,n},$$
which yields the desired explicit formula.
\end{proof}

For $n\geq 1$, let $\gamma_n(x)=\sum_{k=1}^n\gamma_{n,k}x^k$. It follows from~\eqref{gnk-recur} that
$$\gamma_{n+1}(x)=(2n+1)x\gamma_n(x)+x(1-4x)\gamma_n'(x).$$
The first few $\gamma_n(x)$ are $\gamma_0(x)=1,~\gamma_1(x)=x,~\gamma_2(x)=x-x^2,~\gamma_3(x)=x-x^2+3x^3$.
From Proposition~\ref{propFnx}, we see that for any positive even integer $n$, the polynomial $F_n(x)$ is not $\gamma$-positive.

We can now present the third main result of this paper.
\begin{theorem}\label{thmFnx}
The polynomial $F_n(x)$ is semi-$\gamma$-positive. More precisely,
we have $$F_n(x)=\sum_{k=0}^nf_{n,k}x^k(1+x^2)^{n-k},$$
where the numbers $f_{n,k}$ satisfy the recurrence relation
\begin{equation}\label{fnk-recu}
f_{n+1,k}=kf_{n,k}+f_{n,k-1}+4(n-k+2)f_{n,k-2},
\end{equation}
with the initial conditions $f_{0,0}=1$ and $f_{n,0}=0$ for $n\geq 1$.
 Let $f_n(x)=\sum_{k=0}^nf_{n,k}x^k$.
Then
\begin{equation}\label{fxz-EGF}
f(x,z)=\sum_{n=0}^\infty f_n(x)\frac{z^n}{n!}=\sqrt{T(2x,z)},
\end{equation}
where $T(x,z)$ is given by~\eqref{EGF-Tnx}.
\end{theorem}
\begin{proof}
We first consider the grammar~\eqref{Grammar-fap}.
Note that $$D(x)=xyz,~D(yz)=yz(y^2+z^2),~D(y^2+z^2)=4y^2z^2.$$
Set $u=yz$ and $v=y^2+z^2$. Then we have $D(x)=xu,~D(u)=uv$ and $D(v)=4u^2$.
If
\begin{equation}\label{Grammar-gamma-coef}
G_4=\{x\rightarrow xu,~u\rightarrow uv,~v\rightarrow 4u^2\},
\end{equation}
then by induction we see that there exist nonnegative integers $f_{n,k}$ such that
\begin{equation}\label{Dnxfnk-def}
D_{G_4}^n(x)=x\sum_{k=0}^nf_{n,k}u^kv^{n-k}.
\end{equation}
Note that
\begin{align*}
 D_{G_4}^{n+1}(x)
  &=D_{G_4}\left(x\sum_{k=1}^nf_{n,k}u^kv^{n-k}\right)\\
  &=x\sum_{k}f_{n,k}\left(u^{k+1}v^{n-k}+ku^kv^{n-k+1}+4(n-k)u^{k+2}v^{n-k-1}\right).
\end{align*}
By comparing the coefficients of $u^kv^{n+1-k}$, we get~\eqref{fnk-recu}. Moreover, it follows from~\eqref{Grammar-gamma-coef} that
$f_{0,0}=1$ and $f_{n,0}=0$ for $n\geq 1$.
By using~\eqref{Dnxfnk-def}, upon taking $u=yz$ and $v=y^2+z^2$, we get
\begin{equation}\label{Dnxfnk}
D_{G_2}^n(x)=x\sum_{k=0}^nf_{n,k}(yz)^k(y^2+z^2)^{n-k}.
\end{equation}
By comparing~\eqref{Dnxfnk} with~\eqref{DnxFnk}, we get
\begin{equation}\label{Fnxfnk}
F_n(x)=\sum_{k=0}^nf_{n,k}x^k(1+x^2)^{n-k}.
\end{equation}

We now consider a change of the grammar~\eqref{qxy-grammar}. Set $q=\frac{1}{2},~a=x,~b=2u,~c=v$. Then
$$D(x)=xu,~D(u)=uv,~D(v)=4u^2,$$
which are the substitution rules in the grammar~\eqref{Grammar-gamma-coef}. Hence it follows from~\eqref{Dna-EGF} that
\begin{equation*}\label{Dnxfnk2}
\sum_{n=0}^\infty D_{G_4}^n(x)\frac{z^n}{n!}=x\sum_{n=0}^\infty\sum_{k=0}^nf_{n,k}u^kv^{n-k}\frac{z^n}{n!}=xR\left(\frac{2u}{v},vz;\frac{1}{2}\right),
\end{equation*}
which leads to $f(x,z)=R(2x,z;1/2)=\sqrt{T(2x,z)}$.
This completes the proof.
\end{proof}

Combining~\eqref{fxz-EGF} and~\eqref{Fnxfnk},  we immediately get the following result.
\begin{corollary}\label{corfxz}
We have $$F(x,z)=\sum_{n=0}^\infty F_n(x)\frac{z^n}{n!}=\sqrt{T\left(\frac{2x}{1+x^2},(1+x^2)z\right)}.$$
\end{corollary}
It would be interesting to present a combinatorial interpretation of Corollary~\ref{corfxz}.
By using~\eqref{fxz-EGF}, it is not hard to verify that $$\sum_{n=0}^\infty f_{n,n}\frac{x^n}{n!}=\sqrt{\frac{1+\tan x}{1-\tan x}}.$$
It should be noted that the numbers $f_{n,n}$ appear as A012259 in~\cite{Sloane}.
\section{Concluding remarks}
This paper gives a survey of some results related to alternating runs of permutations.
We present a method to construct David-Barton type identities, and based on the survey~\cite{Athanasiadis17}, one can derive several 
David-Barton type identities. Moreover, we introduce the definition of semi-$\gamma$-positive polynomial.
The $\gamma$-positivity of a polynomial $f(x)$ is a sufficient (not necessary) condition for the semi-$\gamma$-positivity of $f(x)$.
In particular, we show that the alternating run polynomials of dual Stirling permutations are semi-$\gamma$-positive.

\end{document}